\documentclass[11pt]{article}
 \usepackage{amsfonts,amssymb,graphicx,amsmath, amsthm}
\usepackage[colorlinks=true,linkcolor=blue,citecolor=blue]{hyperref}

\usepackage{multicol}

\newtheorem{rema}{Remark}
\newtheorem{defi}{Definition}
\newtheorem{lemm}{Lemma}
\newtheorem{theo}{Theorem}
\newtheorem{coro}{Corollary}

\newcommand{\R}[1][]{\ensuremath{{\mathbb{R}^{#1}} }}
\renewcommand{\S}[1][]{\ensuremath{{\mathbb{S}^{#1}} }}

\newcommand{\M}{{\cal M}}

\newcommand{\s}{{\cal S}}
\newcommand{\<}{\langle}
\renewcommand{\>}{\rangle}
\newcommand{\ga}{\gamma}

\newcommand{\al}{\alpha}
\newcommand{\eps}{\epsilon}
\newcommand{\te}{\theta}
\newcommand{\ka}{\kappa}
\newcommand{\la}{\lambda}
\newcommand{\be}{\beta}
\newcommand{\si}{\sigma}
\newcommand{\ove}{\overline}

\newcommand{\co}{{\texttt{cos}\eps}}
\newcommand{\sinu}{{\texttt{sin}\eps}}

\newcommand{\E}{E_N}
\newcommand{\nablaN}{\nabla^{N}}
\newcommand{\ET}{E_T}
\newcommand{\nablaT}{\nabla^{T}}

\renewcommand{\L}{{{\cal R}^{N}}}

\date{}

\title{On the affine Gauss maps of submanifolds of Euclidean space}
\author{ Henri Anciaux\footnote{Universit\'e Libre de Bruxelles, henri.anciaux@gmail.com}, 
Pierre Bayard\footnote{Universidad Nacional Aut\'onoma de M\'exico, bayard@ciencias.unam.mx}}

\begin{document}

\maketitle

\centerline{\textbf {\large{Abstract}}}

\bigskip

{\small It is well known that the space of oriented lines of Euclidean space has a natural symplectic structure. Moreover, given an immersed, oriented hypersurface $\s$ the set of oriented lines that cross $\s$ orthogonally is a Lagrangian submanifold. Conversely, if $\ove{\s}$ an $n$-dimensional family of oriented lines is Lagrangian, there exists, locally, a $1$-parameter family of immersed, oriented, parallel hypersurfaces $\s_t$ whose tangent spaces cross orthogonally the lines of $\ove{\s}.$ The purpose of this paper is to generalize these facts to higher dimension:  to any point $x$ of a submanifold $\s$ of $\R^m$ of dimension $n$ and co-dimension $k=m-n,$ we may associate the affine $k$-space normal to $\s$ at $x.$ Conversely, given an $n$-dimensional family $\ove{\s}$ of affine $k$-spaces of $\R^m$,  we provide certain  conditions granting the local existence of a family of $n$-dimensional submanifolds $\s$ which cross orthogonally the affine $k$-spaces of $\ove{\s}$. We also define a curvature tensor for a general family of affine spaces of $\R^m$ which generalizes the curvature of a submanifold, and, in the case of a $2$-dimensional family of $2$-planes in $\R^4$, show that it satisfies a generalized Gauss-Bonnet formula.

\bigskip

\centerline{\small \em 2010 MSC: 53A07, 	53B25  
\em }

%\medskip

\section*{Introduction}
It is well known that the space $L(\R^{n+1})$ of oriented lines of Euclidean space $\R^{n+1}$ enjoys a natural symplectic structure. The simplest way to understand this is through the identification of $L(\R^{n+1})$  with the tangent bundle of the unit sphere $T \S^{n}.$ The canonical symplectic form of $T\S^{n}$ is $\omega = -d \alpha, $ where $\alpha$
is the tautological form (sometimes also called \em Liouville form\em) defined by the formula $\alpha_{(p,v)} =\< v, d\pi(.)\>$ at $(p,v)\in T\S^n$ ($p\in \S^n,$ $v\in T_p\S^n$), where $\pi : T\S^n \to \S^n$ is the canonical projection. 

Moreover, given an immersed, oriented hypersurface 
$\s \subset \R^{n+1},$ the set of oriented lines that cross $\s$ orthogonally is a Lagrangian submanifold of $L(\R^{n+1})$. This fact has a nice geometric interpretation: generically, a non-flat hypersurface may be locally parametrized by its Gauss map; in this case the Lagrangian submanifold is a section of $T\S^n$ and its generating function is the support function of the hypersurface.
Conversely, given a Lagrangian submanifold $\ove{\s} \subset L(\R^{n+1})$, there exists, locally, a $1$-parameter family of immersed, oriented, parallel hypersurfaces $\s_t$ whose tangent spaces cross orthogonally the lines of $\overline{\s}$.  The situation is very similar if we replace the Euclidean space by a pseudo-Riemannian space form $\mathbb Q^{n+1}_p$ of arbitrary signature $(p,n+1-p)$ and $L(\R^{n+1})$ by the space of geodesics $L(\mathbb Q^{n+1}_p)$ of $\mathbb Q^{n+1}_p$ (see \cite{An}).

The aim of this paper is to generalize these facts to the higher co-dimension submanifolds of Euclidean (or pseudo-Euclidean) space: to any point $x$ of a submanifold $\s$ of $\R^m$ of dimension $n$ and co-dimension $k=m-n,$ we may associate the affine $k$-space normal to $\s$ at $x$. We call this data the \em affine Gauss map \em of $\s.$ Conversely, given an $n$-dimensional family $\overline{\s}$ of affine $k$-spaces of $\R^m$ (a data that we call \em abstract affine Gauss map\em), it is natural to ask when it is the affine Gauss map of some submanifold; in other words: under which condition does there exist locally a family of $n$-dimensional submanifolds $\s$ which cross orthogonally the affine $k$-spaces of $\overline{\s}$? 

For this purpose, we first examine the Grassmannian $\mathcal{Q}$ of affine $k$-spaces. It has a natural bundle structure and we define a natural $1$-form $\al$ on it. This $1$-form is vector-valued rather than real-valued whenever $k>1$, and generalizes the classical tautological form of $T\S^n.$

Next, we consider a map $\overline{\varphi} :\M \to \mathcal{Q}$ defined on an $n$-dimensional manifold $\M$. Pulling back the geometry of $\mathcal{Q}$ via the map $\overline{\varphi}$ induces a natural bundle $\E$ of rank $k$ over $\M$, that we call \em abstract normal bundle, \em
 equipped with a natural connection $\nablaN$  (see the next section for the precise definition). Our main result is the determination of some natural geometric conditions on  $(\E,\nablaN)$ and  $\ove{\varphi}$ which are sufficient to ensure the existence of such "integral" submanifolds $\s.$ The known cases are easily recovered from our result: in the case case of co-dimension $k=1$,  $\mathcal{Q}$ identifies to $\S^{n}$ and the integrability condition is equivalent to the vanishing of the symplectic form $\omega$. If $k >1$ and $(\E,\nablaN)$ is flat, the situation is quite similar.

 We are also able to deal with the case of hypersurfaces in pseudo-Riemannian space-forms (\cite{An}), that we regard as submanifolds of pseudo-Euclidean space $\R^{n+2}$ of co-dimension two contained in a (pseudo)-sphere: in this case the map $\ove{\varphi}$ is valued in the zero section of $\mathcal{Q}$, which makes the integrability condition easier to interpret.

\bigskip

A related problem consists of trying to reconstruct an immersed surface from its linear Gauss map. The difference is that instead of considering affine spaces, one deals with linear spaces, so less information is involved. This issue has been addressed for surfaces ($n=2$) in the following papers: \cite{HO1},\cite{HO2},\cite{W1}-\cite{W3}. To our knowledge this problem is open for higher dimensions $n\geq 3.$

\bigskip

Besides the problem of the prescription of the affine normal spaces of a submanifold, we study some geometric properties of a general congruence of affine spaces. Specifically, we propose a definition of its curvature, and, in the case of a congruence of planes in $\R^4$, we obtain a Gauss-Bonnet type formula. This generalizes to higher dimension and co-dimension similar results obtained in \cite{GK} for a congruence of lines in $\R^3.$ Our treatment is simplified by the use of a formula expressing the curvature of the tautological bundles in terms of the Clifford product (Section \ref{section 8} and Appendix \ref{app curv tautological bundles}); this formula might also be of independent interest.

\bigskip

The paper is organized as follows: in Section \ref{section notations}, we set some notation and state our main results. In Section \ref{Liouville}, we introduce the generalized canonical tautological form, while Sections \ref{section proof th1} and \ref{section construction} are devoted to the proofs of the main theorems. Section \ref{section 7} deals with some special cases of low dimension and co-dimension. The last section is concerned with the general notion of curvature of a congruence. It establishes a Gauss-Bonnet type formula for a congruence of planes in  $\R^4$. Two short appendices end the paper.

\section{Notation and statement of results}\label{section notations}

We denote by $\mathcal{Q}_o=G_{m,n}\subset\Lambda^n\R^m$ the Grassmannian of the oriented linear $n$-planes of Euclidean space $\R^m$ and by  $\mathcal{Q}$ the Grassmannian of the affine oriented $k$-planes where $k$ is such that $n+k=m.$ We have the following identification between $\mathcal{Q}$ and the \emph{tangent tautological bundle} $\tau_T\rightarrow\mathcal{Q}_o:$
\begin{eqnarray}
\mathcal{Q}\simeq\tau_T&=&\{(p_o,v)\in \mathcal{Q}_o\times \R^m,\ v\in p_o\}\\
&=&\{(p_o,v)\in \mathcal{Q}_o\times \R^m,\ p_o\wedge v=0\},
\end{eqnarray}
since a affine oriented $k$-plane of $\R^m$ may be uniquely written in the form $v+p_o^{\perp},$ where $p_o$ is an oriented $n$-plane of $\R^m,$ and $v$ is a vector belonging to $p_o.$ Similarly, we introduce the \emph{normal tautological bundle} $\tau_N\rightarrow\mathcal{Q}_o$ by
$$\tau_N=\{(p_o,v)\in \mathcal{Q}_o\times \R^m,\ v\in p_o^{\perp}\}.$$
We denote by $\pi$ the canonical projection $\mathcal{Q}\rightarrow\mathcal{Q}_o$.

\begin{defi} \em
An  \em abstract affine Gauss map \em  is a map  $\ove{\varphi} :\M \to \mathcal{Q}$, where $\M$ is an $n$-dimensional manifold. Similarly, an \em abstract (linear) Gauss map \em  is a map $\ove{\varphi}_o :\M \to \mathcal{Q}_o.$ 
If $\ove{\varphi} : \M \to \mathcal{Q}$ is an abstract affine Gauss map, then 
$$\ove{\varphi}_o=\pi\circ \ove{\varphi}: \M\rightarrow\mathcal{Q}_o$$ 
is an abstract Gauss map that we shall call the abstract Gauss map \em associated \em to $\ove{\varphi}$.\em
\end{defi}
Given an abstract affine Gauss map $\ove{\varphi}$ and its associated Gauss map $\ove{\varphi}_o$, we consider the bundles $\ET:=\overline{\varphi}_o^*\tau_T$ and $\E:=\ove{\varphi}_o^*\tau_N$ based on $\M$, induced by $\ove{\varphi}_o$ from the tautological  bundles $\tau_T\rightarrow\mathcal{Q}_o$ and $\tau_N\rightarrow\mathcal{Q}_o;$
 these induced bundles are equipped with the connections induced from the natural connections on $\tau_T$ and $\tau_N$, denoted by $\nabla^{T}$ and $\nabla^{N}$ respectively. Since $\tau_T\oplus\tau_N=\mathcal{Q}_o\times\R^m,$ we have
\begin{equation}\label{sum trivial} \ET \oplus \E=\M\times\R^m.
\end{equation}
Moreover, if $\xi_1$ and $\xi_2$ are sections of $\ET\rightarrow \M$ and $\E\rightarrow \M$ respectively, and if $X$ is a vector field along $\M,$ we have 
\begin{equation}\label{covariant derivatives}
\nablaT_X\xi_1=(d\xi_1(X))^T\hspace{.5cm}\mbox{and}\hspace{.5cm}\nablaN_X\xi_2=(d\xi_2(X))^N
\end{equation}
where the superscripts $T$ and $N$ mean that we take the first and the second component of the vectors in the decomposition (\ref{sum trivial}).
\begin{rema}
If the abstract Gauss map $\ove{\varphi}_o$ is in fact the Gauss map of an immersion of $\M$ into $\R^m$ (assuming that an orientation on $\M$ is given), the bundles $\ET$ and $\E$ with their induced connections naturally identify to the tangent and the normal bundles of the immersion, with the Levi-Civita and the normal connections; see also Remark \ref{rmk identification} in Section \ref{section 8}.
\end{rema}

We also consider the $1$-form $\beta:=-\ove{\varphi}^*\alpha \in \Omega^1(\M,\E)$, where $\alpha$ is the canonical $1$-form on $\tau_T,$ with values in $\tau_N,$
$$\al_p (\xi):= d\pi_p (\xi) (v)$$
for all $\xi$ tangent to $\tau_T$ at $p=(p_o,v);$ in this formula $d\pi_p (\xi)\in T_{p_o}\mathcal{Q}_o$ is viewed as a linear map $p_o\rightarrow p_o^{\perp};$ see Section \ref{Liouville} for details.

\begin{theo}\label{first theorem}  Let $\M$ be an $n$-dimensional, smooth and oriented manifold and a map $\ove{\varphi}=(\ove{\varphi}_o,v) :\M \to \mathcal{Q}.$ If $\varphi : \M \to \R^m$ is an immersion whose affine Gauss map is $\overline{\varphi},$ then $s:=\varphi-v\in\Gamma(\E)$ satisfies
\begin{equation}\label{MainEq}
 \nablaN s= \be.
\end{equation}
Conversely, if $s \in \Gamma(\E)$ is a solution of (\ref{MainEq}), then  $\varphi:=s+v,$ if it is an  immersion and preserves orientation, has affine Gauss map $\overline{\varphi}$. The problem of the integration of $\ove{\varphi}$ thus reduces to that of finding a solution $s$ of (\ref{MainEq}) such that $s+v$ is an immersion preserving orientation.
\end{theo}

\begin{rema} We shall call $s$ the \em support function \em of $\varphi$. In the case of co-dimension one, i.e.\ $\E$ has rank one, $s$ identifies with the usual support function of the immersed hypersurface  $\s$.
\end{rema}
\begin{rema}\label{equidistant}
Two solutions of the same problem, i.e.\ which enjoy the same affine Gauss map $\overline{\varphi}$, are equidistant. Indeed, if $\varphi_1$ and $\varphi_2$ are two such solutions, then
$$d(||\varphi_1 -\varphi_2||^2) = 2 \< d\varphi_1 - d\varphi_2 , \varphi_1 - \varphi_2\>.$$
Since 
$ \varphi_i = v + s_i, \, i =1,2,$ we have $\varphi_1 - \varphi_2= s_1-s_2$, which is a normal vector. On the other hand, 
$d\varphi_i, \, i=1,2,$ is tangent. Therefore $\< d\varphi_1 - d\varphi_2 , \varphi_1 - \varphi_2\>$ vanishes.
\end{rema}

It seems difficult to describe the set of solutions of (\ref{MainEq}) in full generality. However, we are able to do so under two additional assumptions which are satisfied in several interesting cases, such as the case of dimension $n=2$ or co-dimension $k=2$. For this purpose we introduce the map
\begin{eqnarray}
\L:\hspace{.5cm} \E&\rightarrow& L\ (\Lambda^2T\M,\E)\label{def Rnabla}\\
\varphi&\mapsto&(\eta\in\Lambda^2T\M\mapsto R^N(\eta)(\varphi)\in \E)\nonumber
\end{eqnarray}
where $R^N=d^{\nablaN}\! \!  \circ \nablaN\in\Omega^2(\M,End(\E))$ is the curvature of the connection $\nablaN.$ 
We write 
$$\be:= \be' + \be'',$$
for all $\be \in \Omega^1(\M,\E)$, according to the direct sum
$$ \E =  Ker\ \L\ \oplus\ (Ker\ \L)^{\perp}$$
($\E$ has a natural metric, pull-back of the canonical metric on $\tau_N$). We shall assume that these two sub-bundles have constant rank and are stable with respect to $\nablaN$.
We moreover need to make the following symmetry assumption: $\overline{\varphi}_o$ is such that
the set of bundle isomorphisms $\Phi:T\M\rightarrow E_T$ satisfying
\begin{equation}\label{nec cond isom sym th}
d\overline{\varphi}_o(X)(\Phi(Y))=d\overline{\varphi}_o(Y)(\Phi(X))\hspace{1cm}\forall\ X,Y\in T\M,
\end{equation}
is not empty. Here $d\overline{\varphi}_o$ is regarded as a map $T\M\rightarrow L(E_T,E_N).$ We then prove the following:

\begin{theo}\label{second theorem}  Let  $\M$ be an $n$-dimensional smooth and oriented manifold and a map $\overline{\varphi} :\M \to \mathcal{Q}.$  We assume that the rank $r$ of $Ker\ \L$ is constant and that the two sub-bundles  $Ker\ \L  $ and $(Ker\ \L)^{\perp}$ of $\E$ are stable with respect to the connection $\nablaN.$ If  moreover (\ref{nec cond isom sym th}) holds and, $\forall X\in T_x \M, \, X\neq 0,$ the linear map
\begin{equation}\label{cond phi immersion Rneq0}
proj\circ d\overline{\varphi}_o(X)\hspace{.3cm}\in\ L(E_T, Ker\ \L)\hspace{.3cm}\mbox{does not vanish,}
\end{equation}
where $proj:E_N \rightarrow Ker\ \L$ stands for the orthogonal projection, then there locally exists an immersion $\varphi: \M \to \R^m$ with affine Gauss map $\overline{\varphi}$ if and only if 
\begin{equation}\label{MainSyst} \left\{ 
\begin{array}{ccc} \overline{\varphi}^*\omega & \in &\Omega^2(\M,(Ker\ \L)^{\perp}) \\  \be''&=&  \nablaN ( (\L)^{-1} (d^{\nablaN} \!\!\be'')),\end{array} \right.
\end{equation}
where $\omega:=d^{\nablaN}\!\!\alpha$ is the canonical 2-form on $\tau_T.$ Moreover, if these equations hold, there exists in fact an $r$-parameter family of local immersions with affine Gauss map $\overline{\varphi},$ which form an $(n+r)$-dimensional Riemannian foliation of $\R^m.$ 
\end{theo}
\begin{rema}
Condition (\ref{nec cond isom sym th}) is obviously necessary for the existence of an immersion $\varphi$ with affine Gauss map $\overline{\varphi},$ since, for $\Phi=d\varphi,$ it expresses the symmetry of the second fundamental form; see also \cite{W1}-\cite{W3} where this condition appears explicitly in the problem of finding immersions of surfaces with prescribed linear Gauss map.
\end{rema}
If $\L=0,$  Condition (\ref{cond phi immersion Rneq0}) is equivalent to saying that $\overline{\varphi}_o$ is an immersion, System (\ref{MainSyst}) reduces to $\overline{\varphi}^*\omega=0$ and we obtain the following result:

\begin{coro}\label{feuilletage cor} 
Let  $\M$ be an $n$-dimensional smooth and oriented manifold and  $\overline{\varphi} :\M \to \mathcal{Q}$ such that $\overline{\varphi}_o$ is an immersion and (\ref{nec cond isom sym th}) holds. If the bundle $(\E,\nablaN)$ is flat, then there locally exists a Riemannian foliation of $\R^m$ whose leaves are integral submanifolds of the abstract affine Gauss map $\overline{\varphi}:\M\rightarrow\mathcal{Q}$ if and only $\overline{\varphi}$ is Lagrangian, i.e. satisfies
\begin{equation}\label{condition lagrangien th}
\overline{\varphi}^*\omega=0.
\end{equation}
\end{coro}

\begin{rema}
This result generalizes the case of a congruence of oriented lines of $\R^{n+1}$ (\cite{An}): in this case $\E$ has rank 1, $\nabla^N$ is flat, so the assumptions of Corollary \ref{feuilletage cor} are satisfied. Condition (\ref{condition lagrangien th}) then amounts to saying that $\overline{\varphi}:\M\rightarrow \mathcal{Q}\simeq T\S^n$ is a Lagrangian map; see Section \ref{Liouville}.
\end{rema}

\begin{rema}
In general, the condition expressing the flatness of  the bundle $(E_N,\nabla^N)$  may be written in the form 
\begin{equation}\label{condition lagrangien2}
\overline{\varphi}_o^*\omega_o=0,
\end{equation}
where $\omega_o\in\Omega^2(\mathcal{Q}_o,End(\tau_N))$ stands for the curvature of $\tau_N\rightarrow\mathcal{Q}_o$. This follows from the fact that the connection $\nabla^N$ is induced from the natural connection on $\tau_N,$ by its very definition (\ref{covariant derivatives}). Similarly to  (\ref{condition lagrangien th}), this condition may be also interpreted as a Lagrangian condition; this is the point of view adopted in \cite{An}. See also Section \ref{subsection space forms}.
\end{rema}

\begin{rema}
There is another interesting case where Theorem \ref{second theorem} applies: if the curvature $\L$ defined in (\ref{def Rnabla}) is injective then Equation (\ref{MainEq}) is solvable if and only if
$$\gamma:=-\overline{\varphi}^*\omega\hspace{.3cm}\in\hspace{.2cm} Im\ \L\hspace{.5cm}\mbox{and}\hspace{.5cm}\nabla((\L)^{-1}(\gamma))=\beta,$$
where $\beta=-\overline{\varphi}^*\alpha.$ The solution is then unique. We mention the following special case: if there exists $\eta\in\Lambda^2T\M$ such that $R^N(\eta):\E\rightarrow \E$ is an isomorphism, then  the map $\L$ is injective. In particular it is the case if $(n,k)=(2,2)$ and the normal curvature is not zero.
\end{rema}
%%%%%%%%%%%%%%%%%%%%%%%%%%%%%%%%%%%%%%%%%%%%%%%%%%%%%%%%%%%%%%%%%
\section{The generalized tautological form} \label{Liouville}
In this section we introduce a natural $\tau_N$-valued one-form on the Grassmannian $\mathcal{Q}$ of the affine $k$-spaces in $\R^m.$ We then show that it generalizes the classical tautological form on $T\S^n,$ and finally that it also satisfies a tautological property.
\subsection{The general construction}\label{Liouville general}
 We introduce the two natural projections
\begin{equation}\label{def pi}
\pi:\begin{array}[t]{ccc}\mathcal{Q}&\rightarrow&\mathcal{Q}_o\\p=(p_o,v)&\mapsto& p_o\end{array}
\end{equation}
and
\begin{equation}\label{def pip}
\pi':\begin{array}[t]{ccc}\mathcal{Q}&\rightarrow&\R^m\\p=(p_o,v)&\mapsto& v.\end{array}
\end{equation}
If $p \in \mathcal{Q}$ and $\xi \in T_p  \mathcal{Q}$ , we have  $d\pi_p(\xi)\in T_{p_o}\mathcal{Q}_o.$ Moreover $T_{p_o}\mathcal{Q}_o$ identifies  naturally with $L(p_o, p_o^\perp)$, the set of linear maps from   $p_o$ to $p_o^{\perp}.$ Hence we can evaluate
 $d\pi_p(\xi)$ at $v,$ therefore getting a vector in $p_o^{\perp}.$ The space $p_o^{\perp}$ is the fiber over $p_o$ of the normal  tautological bundle $\tau_N\rightarrow\mathcal{Q}_o,$ so we obtain a  $\tau_N$-valued $1$-form $\alpha$ on $\mathcal{Q}$ defined as follows:
$$ \al_p (\xi):= d\pi_p (\xi) (v).$$
We can make more explicit the identification $T_{p_o}\mathcal{Q}_o\simeq L(p_o,p_o^{\perp})$. We first set
\begin{eqnarray*}
T_{p_o}\mathcal{Q}_o\subset\Lambda^n\R^m&\simeq& L(p_o,p_o^{\perp})\\
\eta&\mapsto & (v\in p_o\mapsto -i_{p_o}(\eta\wedge v)),
\end{eqnarray*}
where the inner product  $i_{p_o}$ is defined as follows: completing the basis $(e_1,\ldots,e_n)$ of $p_o$ in an orthonormal basis $(e_1,\ldots,e_m)$ of $\R^m,$ we set, for all $1\leq i_1<i_2<\cdots<i_{n+1}\leq m$, 
$$ i_{e_1\wedge\cdots\wedge e_n}(e_{i_1}\wedge e_{i_2}\wedge\cdots\wedge e_{i_{n+1}})= \left\{ \begin{array}{cl} e_{i_{n+1}} &
\mbox{ if } (i_1,i_2,\ldots,i_n)=(1,2,\ldots,n) \\  0 & \mbox{ otherwise. } \end{array} \right.$$
Hence we get 
\begin{equation}\label{def alpha 1}
\alpha_p(\xi)=-i_{p_o}(d\pi_p(\xi)\wedge v).
\end{equation}
This generalized Liouville form $\alpha$ may be defined in a different, but equivalent way: consider the  second projection $\pi':$ we have $d{\pi}'_p(\xi)\in\R^m;$ one may define
\begin{equation}\label{def alphap}
\alpha'_p(\xi)=(d\pi'_p(\xi))^N,
\end{equation}
which is the orthogonal projection of $d\pi'_p(\xi)$ on $p_o^{\perp}.$ We obtain again a $\tau_N$-valued $1$-form $\alpha'$ on $\mathcal{Q}.$ We claim that this $1$-form is nothing but the generalized tautological form we have defined above:  to see this, observe that (\ref{def alphap}) may be written as
\begin{equation}\label{def alpha 2}
\alpha'_p(\xi)=i_{p_o}(p_o\wedge d\pi'_p(\xi)).
\end{equation}
Now, since $p_o\wedge v=0$ for all $(p_o,v)\in\mathcal{Q},$ we get  $\xi=(p'_o,v')\in T_p\mathcal{Q}$ satisfies
$$p'_o\wedge v+p_o\wedge v'=0,$$
so that
$$i_{p_o}(p'_o\wedge v)=-i_{p_o}(p_o\wedge v').$$
Since $p'_o=d\pi_p(\xi)$ and $v'=d\pi'_p(\xi),$ (\ref{def alpha 1}) and (\ref{def alpha 2}) show that $\alpha=\alpha'.$

\subsection{Relation to the classical tautological form}
Here we explain how the form $\al$ is actually a generalization of the canonical tautological form of $T\S^{n}.$ 
We first observe that in the case  $(n,k)=(m-1,1),$ the hypersphere $\S^{n} \subset \R^{n+1}$ is identified with  $\mathcal{Q}_o,$ the Grassmannian of oriented, linear hyperplanes of $\R^{n+1}$.  It follows that $T \S^{n}$ is identified with $\mathcal{Q},$ the Grassmannian of oriented, affine lines of $\R^{n+1}.$ Through these identifications, the natural projection 
 $\pi:T\S^{n}\rightarrow \S^{n}$ corresponds to the map  $\pi:\mathcal{Q}\rightarrow\mathcal{Q}_o$ defined in Section \ref{Liouville general}. 

Using the first definition of $\alpha$ (Definition \ref{def alpha 1}),  we obtain: $\forall p=(p_o,v)\in T\S^{n},$ and  $\forall \xi\in T_p(T\S^{n}),$
 $$\alpha_p(\xi)=-i_{p_o}(* \, d\pi_p(\xi)\wedge v),$$
 where $*d\pi_p(\xi)\in \Lambda^{n}\R^{n+1}$ is the  $n$-vector naturally associated to the vector $d\pi_p(\xi)$ (by the Hodge star operator $*:\R^{n+1}\rightarrow\Lambda^{n}\R^{n+1}$, which is nothing but  the map  identifying $\S^{n}$ with $\mathcal{Q}_o$). It is then straightforward to check that
$$\alpha_p(\xi)=-\<d\pi_p(\xi),v\>p_o,$$
where, in the right hand side term, $p_o$ is seen as a vector of $\R^{n+1}$ (this is a basis of the line normal to the hyperplane represented by $p_o$). Therefore $\alpha_p$ is identified with
$-\<d\pi_p(.),v\>,$ which is, up to the sign, the classical tautological form.

\subsection{The tautological property}

We observe that $\al$, as in the classical case, enjoys a "tautological" property: let $\si$ be a section of $\mathcal{Q}$, i.e. a smooth map $\mathcal{Q}_o \to \mathcal{Q}$ which takes the form $\si(p_o)=(p_o, v(p_o))$. Then we have
$\si^* \al = -v,$ in the following sense: if $\eta \in T_{p_o} \mathcal{Q}_o \simeq L(p_o,p_o^{\perp})$
we have (using the "quantum physics notation"):
\begin{eqnarray*} (\si^* \al)_{p_o} (\eta) &=&\al_{\si(p_o)} (d\si(\eta))\\
 &=&- \< (d\pi \circ d\si)(\eta) | v\> \\
 &=&- \< d(\pi \circ \si)(\eta) |  v\> \\
 & = &- \< \eta | v\>.
\end{eqnarray*}
%%%%%%%%%%%%%%%%%%%%%%%%%%%%%%%%%%%%%%%%%%%%%%%%%%%%%%%%%%%%%%%%%%%%%%

\section{Proof of Theorem \ref{first theorem}}\label{section proof th1}
Given an oriented $m$-dimensional manifold $\M$  and a map
$\overline{\varphi}: \M\rightarrow\mathcal{Q},$
we want to determine under which condition  there exists an immersion $\varphi: \M\rightarrow\R^m$ whose affine Gauss map is $\overline{\varphi}$, 
i.e.\ such that
\begin{equation}\label{eqn phi}
\overline{\varphi}(x)=\varphi(x)+\big(d\varphi_x(T_x\M) \big)^\perp,\quad \forall x \in \M.
\end{equation}
In other words, $\overline{\varphi}(x)$ is the affine $k$-plane normal to $\s:=\varphi(\M)$ at the point $\varphi(x)$.
Assuming that such an immersion exists, we write
\begin{equation}\label{def b phiN}
\varphi(x)=v(x)+s(x), \quad \forall x \in \M,
\end{equation}
according to the direct sum
$$ \R^m =d\varphi_x(T_x\M) \oplus  \big(d\varphi_x(T_x\M) \big)^\perp:= T_x \s \oplus N_x \s.$$
In other words, at the point $\varphi(x)$,  $v(x)$  is   tangent to $\s$  and $s(x)$ is  normal.\\

 Next we observe that $v$ and $s$ may be viewed as sections of  $\ET\rightarrow \M$ and $\E\rightarrow \M$ respectively, and also that
\begin{equation}\label{b proj phib}
v=\pi'\circ\overline{\varphi},
\end{equation}
where $\pi':\mathcal{Q}\rightarrow\R^m$ is the projection defined in (\ref{def pip}): indeed, 
\begin{eqnarray*}
\overline{\varphi}(x)&=&\varphi(x)+N_x \s\\
&=& v(x)+s(x)+N_x \s\\
&=&v(x)+N_x \s,
\end{eqnarray*}
since $s(x)$ belongs to $N_x \s,$ and (\ref{b proj phib}) follows.
Now, differentiating
$$\varphi=v+s,$$
we get
$$d\varphi=dv+ds,$$
and using that $(d\varphi)^N=0$ ($d\varphi$ is tangent to $\s)$) we deduce
$$(ds)^N=-(dv)^N,$$
which is equivalent to 
\begin{equation}\label{eqn}
\nablaN s=-{\overline{\varphi}}^*\alpha = \beta
\end{equation}
in view of (\ref{covariant derivatives}) and (\ref{def alphap}) together with (\ref{b proj phib}). Equation (\ref{eqn}) expresses the equality of two $1$-forms in $ \Omega^1(\M,\E)$ ($1$-forms on $\M$, valued in the bundle $\E\rightarrow \M).$
\\

Conversely, given $s\in\Gamma(E_N)$ a solution of (\ref{eqn}), the differential $d\varphi$ of the function $\varphi:=v+s$  is a 1-form on $\M$ with values in $E_T;$ moreover, if $d\varphi:T\M\rightarrow E_T$ is one-to-one and preserves the orientations, the map $\varphi:\M\rightarrow\R^m$ is an immersion with affine Gauss map
$$\varphi+\left(d\varphi(T\M)\right)^{\perp}=v+{E_T}^{\perp}=\overline{\varphi}.$$
This completes the proof of Theorem \ref{first theorem}.

%%%%%%%%%%%%%%%%%%%%%%%%%%%%%%%%%%%%%%%%%%%%%%%%%%%%%%%%%%%%%%%%%%%%%%%%%

\section{Proof of Theorem \ref{second theorem}}\label{section construction}

\subsection{The formal resolution}
We first proceed to solve Equation (\ref{MainEq}), whose unknown data is $s$; the other objects of the equation are known. Once we have determined a solution $s$ of (\ref{MainEq}), we obtain a solution $\varphi:=v+s$ of the problem (\ref{eqn phi}), where $v$ is defined by (\ref{b proj phib}). The requirement that $\varphi$ be an orientation preserving immersion is an additional constraint on $s;$ this will be studied in Section \ref{section immersion}. Setting 
$\gamma:= d^{\nablaN} \! \! \beta\in\ \Omega^2(\M,\E)$, Equation (\ref{MainEq}) implies
\begin{equation}\label{eqn gal phi}
R^N(\eta)(s)=\gamma(\eta), \quad \forall  \eta\in \Lambda^2T\M,
\end{equation}
which, taking into account the definition of $\L$ (Definition (\ref{def Rnabla})), is equivalent to:
\begin{equation}\label{equation L}
\L(s)=\gamma.
\end{equation}
The rank of $\L$ is assumed to be constant, which implies that 
\begin{equation}\label{splitting E}
\E=Ker\ \L\ \oplus\ (Ker\ \L)^\perp
\end{equation}
is a splitting into two sub-bundles; it is moreover assumed that 
$$\nablaN(Ker\ \L)\subset Ker\ \L,$$
that is that the two sub-bundles $Ker\ \L$ and $(Ker\ \L)^\perp$ are stable with respect to the connection $\nablaN.$ We note that on $Ker\ \L$ the connection $\nablaN$ is flat; thus there exists (locally) a basis $(s_1,\ldots,s_r)$ of orthonormal and parallel sections of $Ker\ \L.$ If $\gamma$ takes value in $Im\ \L,$ the sub-bundle image of the map $\L,$ we may solve (\ref{equation L}): setting $(\L)^{-1}$ the inverse of $\L:(Ker\ \L)^\perp\rightarrow Im\ \L,$ a solution must take the form 
\begin{equation}\label{general solution}
s=\sum_{i=1}^r\lambda_i s_i+(\L)^{-1}(\gamma),
\end{equation}
where $\lambda_1,\ldots,\lambda_r$ are real functions; moreover, since
$$\nablaN s=\sum_{i=1}^rd\lambda_i\ s_i+\nablaN((\L)^{-1}(\gamma))\hspace{.3cm}\in\ Ker\ \L\ \oplus\ (Ker\ \L)^{\perp},$$
setting
\begin{equation}\label{decomposition beta}
\beta=\sum_{i=1}^r\beta_i s_i+\beta'' \hspace{.3cm}\in\ Ker\ \L\ \oplus\ (Ker\ \L)^{\perp},
\end{equation}
we necessarily have (Equation (\ref{MainEq}))
\begin{equation}\label{lambda_i beta_i}
d\lambda_i=\beta_i,\ i=1,\ldots,r
\end{equation}
and the conditions
\begin{equation}\label{compatibility conditions}
d\beta_i=0,\ i=1,\ldots,r\hspace{.5cm}\mbox{and}\hspace{.5cm}\nabla^N((\L)^{-1}(\gamma))=\beta''.
\end{equation}
Conversely, if the compatibility conditions (\ref{compatibility conditions}) hold, then (\ref{general solution}) is a solution of (\ref{MainEq}), where the $\lambda_i$s satisfy (\ref{lambda_i beta_i}). 

A solution, if it exists, is not unique, but it is so modulo adding a parallel section $\sum_{i=1}^r c_i s_i$ depending on $r$ constants $c_1,\ldots,c_r$. 

Finally, we note that the first conditions in (\ref{compatibility conditions}) may be written nicely as follows: if $s$ is a solution of (\ref{MainEq}), we have 
$$R^N(s)=d^{\nablaN}\!\!\beta=-\overline{\varphi}^*\omega$$
where $\omega=d^{\nablaN}\alpha$ is a  $\tau_N$-valued $2$-form on $\mathcal{Q}$. Moreover, since $R^N$ preserves the splitting (\ref{splitting E}) and vanishes on $Ker\ \L,$  $R^N(s)$ is a $2$-form with values in $(Ker\ \L)^{\perp},$ and thus 
\begin{equation}\label{cond lagrangien general}
\overline{\varphi}^*\omega \in \Omega^2(\M,(Ker\ \L)^{\perp}).
\end{equation}  
Using (\ref{decomposition beta}), we see that this condition is equivalent to $d\beta_i=0,\ i=1,\ldots,r.$ 

\subsection{Existence of a solution which is an immersion}\label{section immersion}
A solution $\varphi$ is an immersion if and only if $d\varphi\in\Omega^1(\M,\ET)$ has rank $n$ at every point of $\M.$ Our strategy for proving the existence of such an immersion is the following:
we assume here that the solution $\varphi$ constructed in the previous section is not an immersion at $x_o\in\M$, and we ask if there exists a solution  
$$\varphi+\sum_{i=1}^rc_i s_i,\hspace{1cm} c_1,\ldots,c_r\in\R$$
which is an immersion at $x_o,$ i.e. such that
\begin{equation}\label{dphi}
d\varphi+\sum_{i=1}^rc_id s_i
\end{equation}
has rank $n$ at $x_o$ and preserves the orientation (we recall that $s_1,\ldots,s_r$ are parallel, orthonormal local sections of $Ker\ \L$). To this end, we consider the map
\begin{eqnarray*}
B:\hspace{1cm} \E&\rightarrow& T^*\M\otimes\ET\\
\xi&\mapsto&B(\xi):\ X\mapsto -(d\xi(X))^T ,
\end{eqnarray*}
where in the right hand side $\xi$ is extended to a local section of $\E.$ This is a tensor, since $B(f\xi)=fB(\xi)$ for all smooth functions $f$ on $\M.$  
\begin{rema}
If $\ove{\varphi}$ is the affine Gauss map of an immersion $\varphi':\M\rightarrow\R^m,$ the tensor $B$ naturally identifies with the shape operator $N\M\rightarrow T^*\M\otimes T\M$ of the immersion $\varphi'.$ We therefore call $B$ the \emph{abstract shape operator} of the abstract affine Gauss map $\ove{\varphi}$. Analogously, we may define the \emph{abstract second fundamental form} of  $\ove{\varphi}$ by
\begin{eqnarray*}
h:\hspace{1cm} T\M\times \ET &\rightarrow& \E\\(X,Y)&\mapsto&(dY(X))^N 
\end{eqnarray*}
where in the right hand side $Y$ is extended to a local section of $ \ET.$ This defines a tensor such that
$$ \< h(X,Y), \xi\> = \< B(\xi) (X),Y\>,  \, \, \, \forall( X,Y,\xi) \in T\M \times \ET \times \E.$$
 (See Appendix \ref{app B Gauss map} for the relation between these tensors and the differential of the Gauss map.)
\end{rema}

Since $d s_i=(d s_i)^T$ ($s_i$ is a parallel section of $\E$), the left hand side of (\ref{dphi}) may be written as 
$$d\varphi-\sum_{i=1}^rc_iB(s_i),$$
and the existence of a solution which is an immersion at $x_o$ and preserves the orientation is granted if there exists $\nu\in Ker\, \L$ such that $B(\nu):T\M\rightarrow E_T$ is an isomorphism; indeed, we may then choose $t\in\R$ such that $d\varphi-tB(\nu)$ is an isomorphism which preserves the orientations, and the result follows from the previous discussion.

For sake of simplicity, we first prove the result in the case of a flat normal bundle, which is the context of  Corollary \ref{feuilletage cor}, and only give at the end of the section brief indications for the proof of the general case.

We thus assume that $\overline{\varphi}$ satisfies the assumptions of Corollary \ref{feuilletage cor}: $\overline{\varphi}_o$ is an immersion,
\begin{equation}\label{flat cond}
\overline{\varphi}_o^*\omega_o=0\hspace{.5cm}\mbox{and}\hspace{.5cm}\overline{\varphi}^*\omega=0
\end{equation}
(the first condition is the vanishing of the curvature of $(E_N,\nabla^N),$ i.e. $r=k,$ while the second one is the Lagrangian condition) and there exists a bundle isomorphism $\Phi:T\M\rightarrow E_T$ such that
\begin{equation}\label{nec cond isom sym}
d\overline{\varphi}_o(X)(\Phi(Y))=d\overline{\varphi}_o(Y)(\Phi(X)), \,  \, \, \forall X,Y\in T\M
\end{equation}
(Condition (\ref{nec cond isom sym th})). We first prove the following
\begin{lemm}\label{lem 1 sec 6}
For all $X,Y\in T\M,$ $\nu,\nu'\in E_N,$ 
\begin{equation}\label{sym B(nu)}
\langle B(\nu)(X),B(\nu')(Y)\rangle=\langle B(\nu')(X),B(\nu)(Y)\rangle.
\end{equation}
\end{lemm}

\begin{proof}
This is a consequence of the assumption $\overline{\varphi}_o^*\omega_o=0:$ indeed, since we have
$$B(\nu)(X)=d\overline{\varphi}_o(X)^*(\nu),  \, \, \, \forall X \in T\M, \nu\in E_N$$
 (Appendix \ref{app B Gauss map}), the symmetry condition (\ref{sym B(nu)}) is equivalent to the formula
$$d\overline{\varphi}_o(X)\circ d\overline{\varphi}_o(Y)^*=
d\overline{\varphi}_o(Y)\circ d\overline{\varphi}_o(X)^*, \,  \, \, \forall X,Y\in T\M,$$ 
where the left and the right hand side terms of this identity are regarded as operators on $\overline{\varphi}_o^{\perp};$ on the other hand the following formula holds (Corollary \ref{coro app normal curvature} in Appendix~\ref{app curv tautological bundles}):
$$\overline{\varphi}_o^*\omega_o(X,Y)=d\overline{\varphi}_o(X)\circ d\overline{\varphi}_o(Y)^*-d\overline{\varphi}_o(Y)\circ d\overline{\varphi}_o(X)^*, \, \, \, \forall X,Y\in T\M,$$ where $\overline{\varphi}_o^*\omega_o(X,Y)\in\Lambda^2\overline{\varphi}_o^{\perp}$ is also regarded as a (skew-symmetric) operator on $\overline{\varphi}_o^{\perp}.$
\end{proof}
Theorem \ref{second theorem} follows now from the next lemma:
\begin{lemm}\label{lem 2 sec 6}
There exists $\nu\in E_N$ such that $B(\nu):T\M\rightarrow E_T$ is an isomorphism.
\end{lemm}
\begin{proof}
We first observe that (\ref{nec cond isom sym}) reads 
\begin{equation}\label{nec cond isom sym 2}
\langle \Phi(X),B(\nu)(Y)\rangle=\langle \Phi(Y),B(\nu)(X)\rangle,\,\,\,\forall X,Y\in T\M,\nu\in E_N.
\end{equation}
We consider the inner product in $T\M$ such that $\Phi:T\M\rightarrow E_T$ is an isometry, and observe that (\ref{nec cond isom sym 2}) means that 
$$\Phi^{-1}B(\nu):\ T\M\rightarrow T\M$$ 
is symmetric. Moreover, $\Phi^{-1}B(\nu)$ and $\Phi^{-1} B(\nu')$ commute $\forall \nu,\nu'\in E_N,$ : indeed, using (\ref{nec cond isom sym 2}), we have, $\forall X,Y\in T\M,$
\begin{eqnarray*}
\langle\Phi^{-1} B(\nu)\ \circ\ \Phi^{-1} B(\nu')(X),Y\rangle&=&\langle B(\nu)\left(\Phi^{-1}B(\nu')(X)\right),\Phi(Y)\rangle\\
&=&\langle B(\nu)(Y),B(\nu')(X)\rangle.
\end{eqnarray*}
Now, by (\ref{sym B(nu)}) this last expression is symmetric on $\nu$ and $\nu',$ so $\Phi^{-1}B(\nu)$ and $\Phi^{-1} B(\nu')$  commute. It follows that there exists a basis $(e_1,\ldots,e_n)$ of $T\M$ in which all the symmetric operators $\Phi^{-1} B(\nu),$ $\nu\in E_N,$ are represented by diagonal matrices. Let $(\nu_1,\ldots,\nu_k)$ be a basis of $E_N,$ and set
$$\Phi^{-1} B(\nu_i):=\left(\begin{array}{ccc}
\lambda_1^{(i)}&&0\\
&\ddots&\\
0&&\lambda_n^{(i)}
\end{array}\right),\hspace{1cm}i=1,\ldots,k.$$
If $\nu=\sum_{i=1}^ka_i\nu_i,$ we have
$$\Phi^{-1} B(\nu)=\left(\begin{array}{ccc}
\sum_{i}a_i\lambda_1^{(i)}&&0\\
&\ddots&\\
0&&\sum_ia_i\lambda_n^{(i)}
\end{array}\right).$$
Our goal is then to show that there exist $a_1, \ldots, a_k$ such that $\sum_{i=1}^k a_i\lambda_j^{(i)}\neq 0, \, \forall j, \, 1 \leq j \leq n.$ We first note that
\begin{equation}\label{cond lambdas}
\forall j\in\{1,\ldots,n\},\ \exists i\in\{1,\ldots,k\}\  \mbox{such that}\ \lambda_j^{(i)}\neq 0,
\end{equation}
since, if $\lambda_j^{(i)}=0$ for $i=1,\ldots,k$ then $B(\nu)(e_j)=0$ for all $\nu\in E_N,$ and $d\overline{\varphi}_o(e_j)=0,$ a contradiction since $\overline{\varphi}_o$ is assumed to be an immersion. We consider the linear map
\begin{eqnarray}
\R^k&\rightarrow&\R^n\\
(a_1,\ldots,a_k)&\mapsto& \left(\sum_{i}a_i\lambda_1^{(i)},\ldots,\sum_{i}a_i\lambda_n^{(i)}\right),
\end{eqnarray}
and assume by contradiction that its image is contained in the union of hyperplanes $\cup_{j=1}^n\{x_j=0\};$ it follows that it is contained in  some hyperplane $\{x_{j_o}=0\}.$ This is impossible since, by (\ref{cond lambdas}), $\sum_{i=1}^ka_i\lambda_{j_o}^{(i)}\neq 0$ for some $a_1,\ldots,a_k.$ 
\end{proof}
The arguments above still hold if we replace $E_N$ by $Ker\ \L\subset E_N,$ noting that $\overline{\varphi}_o^*\omega_o\in\Omega^2(\M,End(E_N))$ vanishes on $Ker\ \L$ (thus extending Lemma  \ref{lem 1 sec 6}),  and  that the assumption $B(\nu)(X)=0, \forall \nu \in Ker\ \L$ is not possible  if $Im(d\overline{\varphi}_o(X))$ is not contained in $(Ker\ \L)^\perp$ (for the proof of Lemma \ref{lem 2 sec 6}).

\subsection{The local Riemannian foliation}
Here again, we first assume that $r=k$ (the connection $\nablaN$ on $\E$ is flat), and that $\varphi:\M\rightarrow \R^m$ is a solution which is an immersion at some point $x_0\in \M.$ We show that the family of solutions then locally defines a Riemannian foliation of $\R^m$ at $\varphi(x_o)$. Let us consider $(s_1,\ldots,s_k)$ an orthonormal frame of $\E$ such that the sections $s_1, \ldots, s_k$ 
of $E\rightarrow \M$ are parallel. By the local inverse theorem, the map
\begin{eqnarray*}
\Phi:\ \R^k\times \M&\rightarrow& \R^m\\
((c_1,\ldots,c_k),x)&\mapsto&\varphi(x)+\sum_{i=1}^k c_i s_i(x);
\end{eqnarray*}
is a diffeomorphism from a neighborhood $U\times V$ of $((0,\ldots,0),x_o)$ in $\R^k\times \M$ onto a neighborhood $W$ of $\varphi(x_o)$ in $\R^m.$ If $p_1:\R^k\times \M\rightarrow\R^k$ is the first projection, the map
\begin{eqnarray*}
f:=p_1\circ\Phi^{-1}:W&\rightarrow& U
\end{eqnarray*}
is a Riemannian submersion; indeed, the vectors $\frac{\partial\Phi}{\partial c_j}=s_j,$ $1 \leq j \leq k,$ constitute orthonormal bases of the linear spaces normal to the fibres of $f$: fixing $(c_1,\ldots,c_k)\in U$ and setting $\varphi':=\varphi+\sum_{i=1}^k c_i s_i,$ we have
$$\< \frac{\partial\Phi}{\partial c_j}, d\varphi'(h)\>=
          \< s_j, d\varphi(h)+\sum_{i=1}^k c_i d s_i(h)\>=0,\, \, \forall h\in T\M,$$ since $s_j$ is normal to the immersion $\varphi,$  while $d\varphi(h)$ and $d s_i(h)$ are tangent to $\varphi$ (since $s_i$ is normal and parallel along $\varphi$). This proves the result.
\\

Now, in the general case $r <k$, the argument above still applies and we get the following result: if there exists a solution $\varphi:\M\rightarrow \R^m$   which is an immersion at some point $x_o\in \M,$ then the set of solutions is a local Riemannian foliation of a submanifold of $\R^m$ of dimension $n+r$ (the submanifold of $\R^m$ is the image of the local immersion $\Phi((c_1,\ldots,c_r),x)=\varphi(x)+\sum_{i=1}^r c_is_i(x)$).

%%%%%%%%%%%%%%%%%%%%%%%%%%%%%%%%%%%%%%%%%%%%%%%%%%%%%%%%%%%%%%%%%%%%%%%

\section{Some special cases}\label{section 7}

\subsection{Curves in Euclidean space}\label{subsection curves}
In this section, we give more detail on the case of curves $n=1$. 
According to Theorem \ref{second theorem}, there is no integrability condition in this case. Hence in this case any abstract affine Gauss map $\overline{\varphi}$ should actually be the Gauss map of a family of  curves $\varphi$.

Observe also that the  Grassmannian of affine hyperplanes $\mathcal{Q} $ identifies with $\S^{m-1} \times \R$: at the pair $(\al, \lambda)$ we associate the affine hyperplane $\lambda \al + \al^{\perp}.$ Hence,  given a curve $\overline{\ga}= (\al, \lambda): I \to \S^{m-1} \times \R$, we claim that there exists a $(m-1)$-parameter family of curves $\ga: I \to \R^m$ such
that $\forall t \in I, \ga(t) \in \overline{\ga}$ and $\ga' $ is orthogonal to $\overline{\ga}$, i.e.\ $\ga'$ is collinear to $\al.$

In order to simplify the exposition, we restrict ourselves to the case of curves in $3$-space, i.e. $m=3$ and $k=2$.
We shall use a few technical facts: without loss of generality, we may assume that $\al$ is parametrized by arclength, that we denote by $s$. We orient $\R^3$ and set $\nu:=\al \times \al'$, where $\times$ is the canonical vector product of $\R^3$.
It follows that $(\al, \al', \nu)$ is an orthonormal frame of $\R^3$ along $\al$. In particular, there exist real functions $\la$, $A$ and $B$ on $I$ such that $\ga =\la \al +  A \al' + B \nu= \la \al + s$,
where $s \in \al^{\perp} = Span ( \al', \nu).$ 
Using the
 "Fr\'enet equations" of $\al$ i.e.\  $\al'' = \ka \nu - \al, $ where $\ka$ is the curvature of $\al$ in $\S^2$, and $\nu' = -\ka \al'$ (see \cite{Ku}), we calculate
\begin{eqnarray*}
\ga'  &=&\la' \al + \la \al' +  A' \al' + B' \nu + A \al'' + B \nu' \\
 &=& (\la' - A) \al +   ( \lambda+ A' -  B \ka) \al' + (B' + A \ka)\nu . 
\end{eqnarray*}
Hence the assumption that $\ga'$  is collinear to $\al$ amounts to
$$ \left\{  \begin{array}{ccl} A' &=& B \ka - \la \\  B' & =& -A \ka.\end{array} \right. \leqno{(*)}$$
It is easy to integrate this system: we set $Z:=B+iA$, so that its writes $Z'=i\ka Z - i\la$. The general solution of $Z'=i\ka Z$ is 
$$ Z(t) =C e^{i \te(t)} $$
where $\te(t)= \int_{t} \ka(\tau)d\tau$ and $C$ is a complex  constant. Next we use
the method of variation of constants: for $Z(t) =C(t)  e^{i \te(t)}$ to be a solution of $Z'=i\ka Z - i\la$ we need to have
$ C'(t) e^{i \te(t)}=-i \la$, so that 
$$C(t) = - i \int_t \la(\tau) e^{-i\te(\tau)}d\tau=-\left( \int_t \la(\tau) \sin \te(\tau)d\tau,   \int_t \la(\tau) \cos \te(\tau)d\tau\right) $$
and
$$ (B+iA)(t)=-\left( \int_t \la(\tau) \sin \te(\tau)d\tau,  \int_t \la(\tau) \cos \te(\tau)d\tau\right) e^{i\te(t)}.$$
Thus, $A$ and $B$ are  determined by $\la$ and $\ka$. Since $\ka$ depends on $\al$,  $s$  can be reconstructed from $\overline{\ga}=(\al,\la).$
Moreover, the  $2$-parameter family of curves $\ga$ which are solutions of the previous problem forms, locally, a Riemannian foliation of $\R^3$: to see this, let $\ga_1$ and ${\ga}_2$ two such solutions. Then
$$\ga_1(t) - \ga_2(t) = (A_1(t) - A_2(t)) \al'+ (B_1(t) - B_2(t)) \nu \neq 0,$$
since $(A_1,B_1)$ and $(A_2,B_2)$ are two integral curves of the planar system $(*)$.

Observe moreover that
\begin{eqnarray*}
\frac{d}{dt} || \ga_1 -\ga_2||^2 &=& \frac{d}{dt} \Big(  (A_1 - A_2)^2+ (B_1 - B_2)^2 \Big)\\
&=&  2(A_1 - A_2)(B_1-B_2)\ka+ 2(B_1 - B_2)(A_1 - A_2)(-\ka)\\
&=&0,
\end{eqnarray*}
which means that any two curves of the family of solutions are equidistant: we recover Remark \ref{equidistant}.

 Since $\ga'=(\la'-A) \al$, the curve $\ga$ fails to be regular precisely if $A=\la'$. Clearly, it may happen only for a particular choice of the initial condition in the system $(*)$. Hence, except for a discrete set of values, the curves of the $2$-parameter family of solutions are immersed. 
\bigskip

In higher dimension, one proceeds analogously, introducing an orthonormal frame $(\nu_1, \ldots , \nu_{k-1})$ of
 $Span ( \al, \al') ^{\perp}$ along the curve and introducing the curvatures $\ka_i$ of $\al$ (see \cite{Ku}). Writing
$$\ga = \la \al + A \al' +  \sum_{i=1}^{k-1} B_i \nu_i$$ 
in the orthonormal frame $(\al, \al', \nu_1, \ldots , \nu_{k-1})$, 
and using the Fr\'enet equations, we
 obtain an ordinary differential system, which is linear but non-autonomous, depending on $\la$ and the curve $\al$ via its curvatures functions $\ka_i$, whose unknown functions are $A$ and $B_1, \ldots , B_{k-1}$.
\begin{rema}
Note that we are in fact here in a case where Corollary \ref{feuilletage cor} holds: all the hypotheses of the corollary are trivially satisfied since $n=1.$
\end{rema}
\subsection{Hypersurfaces in space forms}\label{subsection space forms}
We first observe that the whole construction can be generalized \em verbatim \em to the case of Gauss maps of pseudo-Riemannian submanifolds of pseudo-Euclidean space, i.e.\ immersions
of $\R^m$ endowed with the canonical flat pseudo-Riemannian metric of signature $(p,m-p)$
$$  \<\cdot,\cdot\>_p := -dx_1^2 - \ldots - dx_p^2 + dx_{p+1}^2 + \ldots + dx_{m}^2,$$
whose induced metric is non-degenerate. All tangent spaces of a non-degenerate immersion have the same signature, say $(q,n-q)$, so for such immersions we may introduce the affine Gauss map, which is valued in the Grassmannian of affine planes of $(\R^m, \<\cdot,\cdot\>_p)$ with signature $(q,n-q)$. We shall be concerned with submanifolds immersed in the hyperquadrics
$$\mathbb Q^{m-1}_{p,\eps r}=\big\{ x \in \R^{n+2}:\ \<x,x\>^2_p= \eps r^2 \big\},$$
where $\eps=\pm 1$ and $r>0$.

\begin{lemm}\label{lemme hyperquadric}
Let  $\varphi: \M \rightarrow \R^{m}$ an immersion and assume $\M$ is connected. Denote $\ove{\varphi}: \M \to  \mathcal{Q}$ its affine Gauss map. Hence   $\<\varphi, \varphi\>_p=const.$ if and only if $v=\pi' \circ \ove{\varphi}=0$. In other words,
an immersed submanifold of $\R^m$  is in addition contained in a hyperquadric $\mathbb Q^{m-1}_{p,\eps r}$ if and only if  its affine normal spaces (in $\R^m$) are actually vectorial.
\end{lemm}

\begin{proof}
Assume that $ \varphi \in \mathbb Q^{m-1}_{p,\eps r}$. Differentiating the equation $\<\varphi, \varphi\>_p=const$ yields $\<d \varphi, \varphi\>_p=0$, which implies that $\varphi$ is a normal vector. Hence 
$$\ove{\varphi}=\varphi + (Im\ d\varphi)^\perp = (Im\ d\varphi)^{\perp},$$ 
so $\ove{\varphi} \in  \mathcal{Q}_o.$ Conversely, the assumption $\ove{\varphi} \in  \mathcal{Q}_o$ implies that $\varphi \in (Im\ d\varphi)^{\perp}$, i.e.\ $\<d \varphi, \varphi\>_p=0$. Since $\M$ is connected, it follows that $\<\varphi, \varphi\>_p=const$.
\end{proof}

\begin{lemm} If  $\ove{\varphi}: \M \rightarrow  \mathcal{Q}$ is a map satisfying $v=\pi' \circ \ove{\varphi}=0$, then $\beta$ vanishes. It follows that $s$ must be parallel with respect to $\nabla^{N}.$
\end{lemm}

We assume now that $k=2$: let $\ove{\varphi}: \M \rightarrow  \mathcal{Q}_o=G(n,n+2)$. We write $\ove{\varphi}= (e_1\wedge e_2)^\perp$, where 
$(e_1,e_2)$ is an orthonormal frame of $\ove{\varphi}^\perp$ with $\<e_1,e_1\>_p=1$ and $\<e_2,e_2\>_p=\eps,$ $\eps=\pm 1.$ Hence if $\eps=1$ the plane $e_1 \wedge e_2$  has positive definite metric, and if $\eps=-1$ it has indefinite metric.

If $\varphi=s$ is a section of $\E$ with $\<s, s\>=1$, there exists $\te \in C(\M)$ such that
$\varphi = \co(\te) e_1 + \sinu (\te) e_2$, where $(\co,\sinu)=(\cos ,\sin)$ (resp.\ $(\cosh,\sinh)$) if $\eps=1$ (resp.\ $\eps=-1$).
Observe that a unit normal vector along $\varphi$ is given $N:=-\eps \, \sinu (\te) e_1 + \co(\te) e_2$.

By the lemma above, the fact that $\ove{\varphi}$ is the affine Gauss map of $\varphi$ is equivalent to the vanishing  of $(d\varphi)^N.$
We have
$$d\varphi = (-\eps \, \sinu (\te) e_1 + \co(\te) e_2) d\te + \co (\te) de_1 + \sinu (\te) de_2.$$
Using the fact that $\<de_1,e_1\>_p$ and $\<de_2,e_2\>_p$ vanish and $\<de_1, e_2\>_p=-\<e_1, de_2\>_p$, we calculate
\begin{eqnarray*} (d\varphi)^N&=&  (-\eps \, \sinu (\te) e_1 + \co(\te) e_2) d\te + \co (\te) \<de_1,e_2\>_p e_2 + \eps \sinu (\te) \<de_2,e_1\>e_1\\
 & = & (-\eps \, \sinu (\te) e_1 + \co(\te) e_2) (d\te - \eps \<de_2,e_1\>_p)\\
&=& (d\te - \eps \<de_2,e_1\>_p)N.
\end{eqnarray*}
The latter vanishes if and only if $d\te = \eps \< de_2,e_1\>_p$. Such a map $\te \in C(\M)$ exists (at least locally) if and only if the one-form $\<de_2, e_1\>_p$ is closed. This is exactly saying that the immersion $\ove{\varphi}: \M \rightarrow  \mathcal{Q}_o=G(n,n+2) \simeq G(2,n+2)$ is Lagrangian with respect to the natural symplectic structure of $G(2,n+2)$. 

\begin{rema}
This result may be also obtained as follows: we first note that the natural symplectic structure on $\mathcal{Q}_o\simeq G(2,n+2)$ may be interpreted as the curvature form $\omega_o\in\Omega^2(\mathcal{Q}_o,End(\tau_N))$ of $\tau_N\rightarrow \mathcal{Q}_o:$ this curvature form is indeed a $2$-form with values in the skew-symmetric operators acting on $\tau_N,$ and may thus be naturally identified to a real form (since here the rank of $\tau_N$ is 2). Finally, the existence of a non-trivial parallel section of $\ove{\varphi}_o^{*}\tau_N\rightarrow \M$ is obviously equivalent to the vanishing of the curvature  $\ove{\varphi}_o^{*}\omega_o,$ since the rank of the bundle is 2.
\end{rema}

In the case that $\overline{\varphi}:\M\rightarrow\mathcal{Q}_o$ is Lagrangian, given a solution $\varphi =  \co(\te) e_1 + \sinu (\te) e_2$, the other ones take the form
\begin{eqnarray*} \varphi_t &=&   \co(t) \varphi+ \sinu(t) N \\
 &=& \co(\te+t) e_1 + \sinu (\te+t) e_2,
\end{eqnarray*}
where $t$ is a real constant (if $\eps=1$, $t $ is defined $mod \, 2\pi$).
We claim that if $\ove{\varphi}$ is an immersion, then $\varphi_t$ is an immersion except for at most $n=m-2$ 
distinct values of $t $
(distinct  $mod \, \pi$ if $\eps=1$).
 To see this, observe first that, taking into account that $d\te= \eps \<de_2,e_1\>_p = -\eps \<de_1,e_2\>_p$,
$$ d\varphi_t = (d\varphi_t)^T = \co(\te+t) (de_1)^T +   \sinu(\te+t) (de_2)^T.$$
Hence if there exists more than $n$ (the dimension of $T\M$) distinct values of $t$ such that $Ker (d\varphi_t) \neq \{0\}$, there must be a pair $(t,t')$ (with $t \neq t' \, mod \, \pi$ if $\eps=1$ and $t \neq t'$ if $\eps=-1$) such that
$$(d\varphi_t)(X)= (d\varphi_{t'})(X)=0$$ 
for some non vanishing vector $X \in T\M$. Therefore
$$(de_1)^T(X)=  (de_2)^T (X)=0$$
which implies that $\ove{\varphi}= (e_1\wedge e_2)^\perp$ is not an immersion since
$$ d(e_1\wedge e_2) =de_1 \wedge e_2 + e_1 \wedge de_2 =(de_1)^T \wedge e_2 + e_1 \wedge (de_2)^T.$$
 In particular ($\eps=1$), we obtained \cite{An}:
\begin{theo} 
Let us consider $\ove{\varphi}_o:\M\rightarrow G(2,n+2),$ a $n$-parameter family of geodesic circles of $\S^{n+1}.$ We moreover assume that it is an immersion. There exists a hypersurface of $\S^{n+1}$ orthogonal to the family $\ove{\varphi}_o$ if and only if the Lagrangian condition $\ove{\varphi}_o^{*}\omega_o=0$ holds. When this is the case, there is a one-parameter family of such integral hypersurfaces which form a parallel family (with at most $n$ singular leaves).
\end{theo}
\begin{rema}
It appears here that condition (\ref{nec cond isom sym th}) is not necessary: it may be proved that it is indeed a consequence of the other hypotheses (immersions of co-dimension 1 in space forms).
\end{rema}

\subsection{Submanifolds with flat normal bundle in space forms}
The previous section may be generalized to higher co-dimension ($k\geq 3$) as follows: suppose that $\ove{\varphi}:\M\rightarrow\mathcal{Q}$ is an immersion such that $v\equiv 0$ and (\ref{nec cond isom sym th}) holds; then there exists a submanifold $\s$ of $(\R^m,\<\cdot,\cdot\>_p)$ with normal affine Gauss map $\ove{\varphi}$ and flat normal bundle if and only if 
$$\ove{\varphi}^*\omega_o=0$$ 
where $\omega_o\in\Omega^2(\mathcal{Q}_o,End(\tau_N))$ is the curvature form of $\tau_N\rightarrow\mathcal{Q}_o.$ This is a consequence of Corollary \ref{feuilletage cor} and of the fact that (\ref{condition lagrangien th}) obviously holds in that case (since $\beta=-\ove{\varphi}^*\alpha=0$ when $v=0$). In view of Lemma \ref{lemme hyperquadric} the submanifold is contained in some $\mathbb Q^{m-1}_{p,\eps r}$. 

\subsection{Surfaces with non-vanishing normal curvature in 4-dimen\-sional space forms}
We assume here that $n=2$ and $k=3,$ and suppose that $\ove{\varphi}:\M\rightarrow\mathcal{Q}$ is an immersion such that $v\equiv 0,$ (\ref{nec cond isom sym th}) holds, and, in the splitting (\ref{splitting E}),
$$rank\ Ker\ \L=1\hspace{.5cm}\mbox{and}\hspace{.5cm}rank\ (Ker\ \L)^{\perp}=2.$$
We note that $\beta$ and $\gamma$ are zero here (since $v=0$), and thus that the necessary and sufficient conditions (\ref{compatibility conditions}) hold in that case. In view of (\ref{general solution}), a solution $s$ is just a parallel section of $Ker\ \L.$  We show independently of Section \ref{section construction} that such a section exists: we fix a local section $\sigma\in\Gamma(Ker\ \L)$ which does not vanish, and determine a function $f$ 
such that $s=e^f \si$ is parallel. The condition $\nablaN s=0$ reads
$$df=-\mu,$$
where $\mu$ is such that $\nablaN \si=\mu \otimes \si.$ The form $\mu$ is closed: this is a direct consequence of $R^\nabla(X,Y)\si=0$ for all $X,Y\in T\M$ ($\si$ is a section of $Ker\ \L$). The solution $s$ is unique, up to homothety. It defines an immersion with normal affine Gauss map $\ove{\varphi}$ which is valued in a hyperquadric ($v=0$). Its normal bundle in the hyperquadric identifies to $(Ker\ \L)^{\perp}$ and its normal curvature to the restriction of $\L$ to this bundle; thus, its normal curvature does not vanish. 

%%%%%%%%%%%%%%%%%%%%%%%%%%%%%%%%%%%%%%%%%%%%%%%%%%%%%%%%%%%%%%%
\section{The curvature of a congruence}\label{section 8}
The aim of this section is to introduce the curvatures of a general family of affine $k$-spaces in $\R^m.$ We begin with a formula expressing the curvatures of a submanifold in terms of its Gauss map, we then propose a general definition, and we finally establish a Gauss-Bonnet type formula for a 2-parameter family of affine planes in $\R^4.$ The results of this section generalize results in \cite{GK} concerning a general congruence of lines in $\R^3.$
\subsection{The curvature of a submanifold in terms of its Gauss map}
The curvature tensor of the tautological bundles on the Grassmannnian may be easily described if we consider the fibres of the bundles $T\mathcal{Q}_o$, $\tau_T$ and $\tau_N$ on $\mathcal{Q}_o$ as subsets of the Clifford algebra  $Cl(\R^m),$ endowed with the bracket $[\cdot,\cdot]$ defined by 
$$[\eta,\eta']=\frac{1}{2}\left(\eta\cdot\eta'-\eta'\cdot\eta\right), \,\,\, \forall \eta,\eta'\in Cl(\R^m),$$
where the dot $"\cdot"$ is the Clifford product in $Cl(\R^m).$ See e.g.\ \cite{F} for the basic properties of the Clifford algebras.
\begin{theo}\label{th expr curv}
The curvature tensor of the tautological bundles on the Grassmannnian is given by
\begin{equation}\label{curvature double bracket}
R(u,v)\xi=\epsilon_n\, [\xi,[u,v]]
\end{equation}
for all $u,v\in T_{p_o}\mathcal{Q}_o$ and $\xi\in {\tau_T}_{p_o}\oplus{\tau_N}_{p_o},$ where $\epsilon_n=(-1)^{\frac{n(n+1)}{2}+1},$ or, more concisely, 
\begin{equation}
R(u,v)=\epsilon_n\, [u,v]\hspace{0.3cm}\in\hspace{.3cm}\Lambda^2p_o\oplus\Lambda^2p_o^\perp,
\end{equation}
where an element of $\Lambda^2p_o\oplus\Lambda^2p_o^\perp$ is regarded as a skew-symmetric operator on $p_o\oplus p_o^{\perp}$ (by the natural identification, or, equivalently, using the bracket, as in (\ref{curvature double bracket})). 
\end{theo}
The proof of this theorem is given in Appendix \ref{app curv tautological bundles}. With this result, the curvature of the Levi-Civita and the normal connections of an immersed submanifold $\varphi:\M\rightarrow\R^m$ with Gauss map $\overline{\varphi}_o:\M\rightarrow\mathcal{Q}_o$ is
$$R^{\nabla^T\oplus\nabla^N}(u,v)=\epsilon_n\, [d{\ove{\varphi}_o}_x(u), d{\overline{\varphi}_o}_x(v)], \,\,\,\forall u,v\in T_x\M;$$
in this last expression the bracket 
$$[d{\overline{\varphi}_o}_x(u), d{\overline{\varphi}_o}_x(v)]\hspace{0.3cm}\in\hspace{.3cm}\Lambda^2\overline{\varphi}_o(x)\oplus \Lambda^2\overline{\varphi}_o(x)^{\perp}$$
is a skew-symmetric operator on ${\overline{\varphi}_o(x)}\oplus{\overline{\varphi}_o(x)}^{\perp}\simeq T_x\varphi(\M)\oplus N_x\varphi(\M).$ This is a consequence of the following simple remark:

\begin{rema}\label{rmk identification}
If $\varphi:\M\rightarrow\R^m$ is an immersion with affine Gauss map $\ove{\varphi},$ then the bundles $\ET\rightarrow \M$ and $\E \rightarrow \M$ naturally identify with the tangent and the normal bundle of $\s:=\varphi(\M)$ respectively (we should say, with the pull-backs of these bundles on $\M$); moreover, under these identifications and by (\ref{covariant derivatives}), the connections $\nabla^{T}$ and $\nabla^{N}$ identify with the Levi-Civita and the normal connections of $\s.$ In particular, the Gauss and the normal curvature tensors of $\s$ naturally identify with the curvature tensors of $\nabla^{T}$ and $\nabla^{N}$ respectively; thus, these tensors identify with the pull-backs of the curvature tensors of the tautological bundles $\tau_T\rightarrow\mathcal{Q}_o$ and $\tau_N\rightarrow\mathcal{Q}_o$ by the Gauss map $\ove{\varphi}_o:\M\rightarrow\mathcal{Q}_o.$
\end{rema}

\subsection{Generalized curvature tensor of a congruence}
\begin{defi}\label{def general curvature}
The curvature of a congruence $\overline{\varphi}:\M\rightarrow\mathcal{Q}$ is the curvature of the tautological bundle $\ET\oplus\E$; it is given by
$$R(u,v):=\epsilon_n\ [d{\ove{\varphi}_o}_x(u), d{\ove{\varphi}_o}_x(v)]\in \Lambda^2\ove{\varphi}_o(x)\oplus \Lambda^2\overline{\varphi}_o(x)^{\perp}, \,\,\, \forall u,v\in T_x\M:$$
 this is a 2-form on $\M$ with values in the skew-symmetric operators of  the fibres of $\ET\oplus\E$.
\end{defi}
We note that, even in the case of a congruence of lines in $\R^3,$ we cannot attach a number to the curvature tensor of a congruence, since we do not have a natural metric, nor a natural volume form, on $\M$; nevertheless, if a point $\lambda\in\ove{\varphi}_o(x)^\perp$ is additionally given, we can define a natural inner product on $T_x\M,$ and then define real valued curvatures: they will depend on $x$ and $\lambda$. To motivate the definition of this inner product, we first assume that the congruence $\ove{\varphi}:\M\rightarrow\mathcal{Q}$ is integrable, i.e.\ that there exists an immersion $\varphi=v+s$ of $\M$ in $\R^m$ with Gauss map $\ove{\varphi};$ we note that
$$d\varphi=(d\varphi)^T=(dv)^T+(ds)^T=\nablaT v-B(s),$$
where $B$ was introduced before. Thus the map
$$\nablaT v-B(s):T_x\M\rightarrow T_{\varphi(x)}\s$$ 
is the natural identification between $T_x\M$ and $T_{\varphi(x)}\s,$ and the natural inner product induced on $T_x\M$ is such that this map is an isometry. We now come back to the general case, and only assume that a congruence $\overline{\varphi}:\M\rightarrow\mathcal{Q}$ is given ($\overline{\varphi}$ is not necessarily integrable); if $\lambda$ belongs to the plane $\ove{\varphi}_o(x)^\perp,$ and if we assume that the map
$$\nablaT v-B(\lambda):T_x\M\rightarrow \R^m$$
is one-to-one, we can endow $T_x\M$ with the metric such that this map is an isometry. Now, at each point $(x,\lambda)$ belonging to the congruence, we have:
\begin{itemize}
\item[---] a curvature tensor $R:\ \Lambda^2T_x\M\rightarrow \Lambda^2\overline{\varphi}_o(x)\oplus \Lambda^2\overline{\varphi}_o(x)^{\perp};$

\item[---] metrics on $\overline{\varphi}_o(x),$ $\overline{\varphi}_o(x)^{\perp}$ and $T_x\M.$
\end{itemize}

In the case of a congruence of lines in $\R^3,$ a real valued curvature may thus be attached to each point $(x,\lambda)$ of the congruence: this is the number $K(x,\lambda)$ such that
$$R(u,v)=\big( K(x,\lambda)dA_{x,\lambda}(u,v)\big)\ \ove{\varphi}_o(x)$$ 
$\forall u,v\in T_x\M,$ where $dA_{x,\lambda}$ is the area form on $\M$ induced by the metric at $(x,\lambda)$ (observe that $\Lambda^2\ove{\varphi}_o(x)=\R\ \ove{\varphi}_o(x)$ and $\Lambda^2\ove{\varphi}_o(x)^{\perp}=0$ in that case); it is not difficult to see that this definition coincides with the curvature introduced in \cite{GK}.
\\

In the case of a congruence of planes in $\R^4,$ we may  define, in a similar way,  generalized Gauss  and normal curvatures at each point $\lambda\in\ove{\varphi}_o(x)^\perp$ of the congruence. We first introduce the following 2-forms on $\M$:
$$ \hspace{-4em} \omega_T(u,v):=\< R(u,v),\ove{\varphi}_o(x)\>$$ 
and
$$ \hspace{4em}\omega_N(u,v):=\< R(u,v),\ove{\varphi}_o(x)^{\perp}\> , \quad \forall u,v\in T_x\M,$$
 where $\< \cdot, \cdot\>$ is here the canonical scalar product on $\Lambda^2\R^4$. Equivalently: 
$$R=\omega_T \ove{\varphi}_o(x)\ +\ \omega_N \ove{\varphi}_o(x)^\perp$$
(here $\Lambda^2\ove{\varphi}_o(x)=\R\ \ove{\varphi}_o(x)$ and $\Lambda^2\overline{\varphi}_o(x)^{\perp}=\R\ \ove{\varphi}_o(x)^{\perp}$). If $\lambda\in\overline{\varphi}_o(x)^\perp$ is given, we define the generalized Gauss and normal curvatures $K(x,\lambda)$ and ${K_N}(x,\lambda)\in\R$ by the formulas
$$\omega_T=K(x,\lambda)dA_{x,\lambda}\hspace{.5cm}\mbox{and}\hspace{.5cm}\omega_N={K_{N}}(x,\lambda)dA_{x,\lambda},$$
where $dA_{x,\lambda}$ is the area form on $\M$ induced by the metric at $(x,\lambda)$ introduced above. By the very definition of these quantities, if $\varphi=v+s$ is an immersion with Gauss map $\ove{\varphi}$ (assuming thus that the congruence is integrable), $K(x,s(x))$ and ${K_{N}}(x,s(x))$ coincide with the Gauss and the normal curvatures of the immersion at $x.$ 

%%%%%%%%%%%%%%%%%%%%%%%%%%%%%%%%%%%%%%%%%%%%%%%%%%%%%%%%%%
\subsection{A Gauss-Bonnet formula for a congruence of planes in~$\R^4$}

\begin{theo}
Let $\M$ be a closed and oriented surface, and $\overline{\varphi}:\M\rightarrow\mathcal{Q}$ a congruence of affine $2$-planes in $\R^4.$ Then
$$\int_{\M}\omega_T=2\pi\chi (\ET )\hspace{1cm}\mbox{and}\hspace{1cm}\int_{\M}\omega_{N}=2\pi  \chi(\E),$$
where $\chi(\ET)$ and $\chi(\E)$ are the Euler characteristics of the tautological bundles induced on $\M$ by the Gauss map. Moreover
$$\chi(\ET)=\deg g_1+\deg g_2\hspace{.5cm}\mbox{and}\hspace{.5cm}\chi(\E)=\deg g_1-\deg g_2,$$
where $\overline{\varphi}_o=(g_1,g_2)$ in the natural identification $\mathcal{Q}_o\simeq \S^2\left(\frac{\sqrt{2}}{2}\right)\times \S^2\left(\frac{\sqrt{2}}{2}\right).$
\end{theo}

\begin{rema}
If the congruence is integrable we obtain the classical formulas
$$\int_{\M}KdA=2\pi  \chi(T\s)\hspace{1cm}\mbox{and}
\hspace{1cm}\int_{\M}K_NdA=2\pi   \chi(N\s)$$
where $dA$ is the area form induced by the immersion.
\end{rema}

\begin{proof}
We only prove the claims about the bundle $\tau_T,$ since the proofs concerning the bundle $\tau_N$ are very similar. Let us denote by $\Omega\in\Omega^2(\M,End(\tau_T))$ the curvature form of the tautological bundle $\tau_T\rightarrow\mathcal{Q}_o;$ by Theorem \ref{th expr curv}, it is given by
$$\Omega(u,v)(\xi)=\left[\xi,[u,v]\right], \quad \forall u,v\in T_{p_o}\mathcal{Q}_o, \, \xi\in{\tau_T}_{p_o}.$$
  It is easy to check that its Pfaffian is given by
$$Pf(\Omega)(u,v):=\<\Omega(u,v)(e_2),e_1\>=\<[u,v],e_1\wedge e_2\>, \quad \forall u,v\in T_{p_o}\mathcal{Q}_o,$$
where $(e_1,e_2)$ is a positively oriented and orthonormal basis of the fibre ${\tau_T}_{p_o},$ and where we use the same notation $\<\cdot,\cdot\>$ to denote the scalar products on $\R^4$ and on $\Lambda^2\R^4.$ Thus the Euler class of the tautological bundle $\tau_T\rightarrow\mathcal{Q}_o$ is
$$e(\tau_T)=\frac{1}{2\pi}\langle[u,v],e_1\wedge e_2\rangle,$$
and the Euler class of the induced bundle $\ET$ is
$$e(\ET)=\overline{\varphi}_o^*e(\tau_T)=\frac{1}{2\pi}\omega_T,$$
by the very definition of $\omega_T.$ Thus
$$\chi(\ET):=\int_{\M}e(\ET)=\frac{1}{2\pi}\int_M\omega_T,$$
which proves the first claim in the theorem. The last claim is a direct consequence of the formula
\begin{equation}\label{omegaT omegai}
\omega_T=g_1^*\omega_1+g_2^*\omega_2;
\end{equation}
this last formula holds, since easy calculations give the formulas
\begin{eqnarray}
\langle [u,v],p_o\rangle&=&\langle [u^+,v^+],p_o^+\rangle+\langle [u^-,v^-],p_o^-\rangle\nonumber\\
&=&{\omega_1}_{p_o^+}(u^+,v^+)+{\omega_2}_{p_o^-}(u^-,v^-),\label{omegaT omegai abstract}
\end{eqnarray}
$\forall p_o\in\mathcal{Q}_o,$ $u,v\in T_{p_o}\mathcal{Q}_o,$ where $p_o=p_o^++p_o^-,$ $u=u^++u^-,$ $v=v^++v^-$ in the splitting 
$$\Lambda^2\R^4=\Lambda^+\R^4\oplus\Lambda^-\R^4,$$ 
and where $\omega_1$ and $\omega_2$ are the area forms of $\S^2(\sqrt{2}/2)\subset \Lambda^+\R^4$ and $\S^2(\sqrt{2}/2)\subset \Lambda^-\R^4$ respectively; the pull-back of (\ref{omegaT omegai abstract}) by $\overline{\varphi}_o=(g_1,g_2)$ finally gives (\ref{omegaT omegai}).
\end{proof}

%%%%%%%%%%%%%%%%%%%%%%%        APPENDIX             %%%%%%%%%%%%%%%%%%%%%%%%%%%%%%%%%%%
\appendix
\section{The curvature of the tautological bundles}\label{app curv tautological bundles}
We prove here Theorem \ref{th expr curv}. We only deal with the case of the tautological bundle $\tau_T,$ since the proof for the bundle $\tau_N$ is very similar. We consider $u,v\in \Gamma(T\mathcal{Q}_o)$ and $\xi\in\Gamma(\tau_T)$ such that, at $p_o,$
\begin{equation}\label{properties u v xi}
dv(u)-du(v)=0\hspace{1cm}\mbox{and}\hspace{1cm}\nablaT\xi=0.
\end{equation}
Since, in a neighborhood of $p_o,$ 
\begin{equation}\label{property dxi}
\nabla_u\xi=d\xi(u)-{d\xi(u)}^N=d\xi(u)-u(\xi),
\end{equation}
where in this last expression $u\in T_{p}\mathcal{Q}_o$ is regarded as an element of $L({\tau_T}_p,{\tau_N}_p)$, we get
\begin{eqnarray}
R^{T}(u,v)\xi&=&\nablaT_u\nablaT_v\xi-\nablaT_v\nablaT_u\xi\nonumber\\
&=& \nablaT_u(d\xi(v)-v(\xi))-\nabla_v(d\xi(u)-u(\xi))\nonumber\\
&=&\big( d(u(\xi))(v)-d(v(\xi))(u)\big)^T,\label{formula R u v}
\end{eqnarray}
where we used that $d\circ d\ \xi=0.$ The superscript $T$ means that we take the component of the vector belonging to $p_o.$ We will need the following 
\begin{lemm} \label{lemmeClifford}
For all $u\in T_{p}\mathcal{Q}_o$ and $\xi\in p,$
$$u(\xi)=-\frac{\epsilon_n}{2}\big(\xi\cdot p\cdot u+u\cdot p\cdot\xi\big)$$
where in the left hand side $u$ is considered as a linear map $p\rightarrow p^\perp$ and in the right hand side $\xi,p,u$ are viewed as elements of the Clifford algebra $Cl(\R^m)$; the dot "$\cdot$" stands for the Clifford product in $Cl(\R^m).$ 
\end{lemm}
\begin{proof}[Proof of Lemma \ref{lemmeClifford}] We write
$$u=\sum_{i=1}^ne_1\wedge \ldots\wedge u(e_i)\wedge\ldots\wedge e_n=\sum_{i=1}^ne_1\cdot \ldots\cdot u(e_i)\cdot\ldots\cdot e_n,$$
and we compute
\begin{eqnarray*}
\xi\cdot p\cdot u&=&\sum_i\xi\cdot e_1\cdot\ldots\cdot e_n \cdot e_1\cdot \ldots\cdot u(e_i)\cdot\ldots\cdot e_n\\
&=&\sum_i\xi\cdot e_1\cdot\ldots\cdot e_n \cdot e_1\cdot \ldots\cdot e_i\cdot\ldots\cdot e_n\cdot(-e_i)\cdot u(e_i)\\
&=&\sum_i\epsilon_n\ \xi\cdot e_i\cdot u(e_i)
\end{eqnarray*}
since $(e_1\cdot\ldots\cdot e_n)^2=-\epsilon_n.$ Similarly, we compute
$$u\cdot p\cdot\xi=\sum_i\epsilon_n\ u(e_i)\cdot e_i\cdot\xi,$$
and thus get
\begin{eqnarray*}
\xi\cdot p\cdot u+u\cdot p\cdot\xi&=&\epsilon_n\sum_i\left(\xi\cdot e_i\cdot u(e_i)+u(e_i)\cdot e_i\cdot\xi\right)\\
&=&\epsilon_n\sum_i\left(\xi_ie_i\cdot e_i\cdot u(e_i)+u(e_i)\cdot e_i\cdot\xi_ie_i\right)\\
&=&-2 \epsilon_n u(\xi),
\end{eqnarray*}
which is the required formula.
\end{proof}
Using Lemma \ref{lemmeClifford}, we get
\begin{eqnarray*}
d(u(\xi))(v)&=&-\frac{\epsilon_n}{2}\Big(d\xi(v)\cdot p_o\cdot u+\xi\cdot v\cdot u+\xi\cdot p_o\cdot du(v)\\&&
\quad \quad +du(v)\cdot p_o\cdot\xi+u\cdot v\cdot\xi+u\cdot p_o\cdot d\xi(v)\Big).
\end{eqnarray*}
Moreover, by (\ref{property dxi}) and since $\nabla\xi=0$ at $p_o,$ we have 
$$d\xi(v)=v(\xi)=-\frac{\epsilon_n}{2}\big(\xi\cdot p_o\cdot v+v\cdot p_o\cdot\xi\big),$$
 and we thus obtain an expression of $d(u(\xi))(v)$ in terms of the Clifford products of $u,v,\xi,p_o$ and $du.$ Switching $u$ and $v$ we get a similar formula for  $d(u(\xi))(v).$ Plugging these two formulas in (\ref{formula R u v}) 
and using the first identity in (\ref{properties u v xi}) we may then easily get
$$R(u,v)\xi=\epsilon_n\big([\xi,[u,v]]\big)^T$$
(using moreover that $u\cdot p_o=-p_o\cdot u,$ $v\cdot p_o=-p_o\cdot v$ and $p_o\cdot p_o=-\epsilon_n$); this gives the result since $[\xi,[u,v]]$ is in fact a vector belonging to $p_o$: indeed, it is easy to check that $[u,v]$ belongs to $\Lambda^2p_o\oplus\Lambda^2p_o^\perp$ if $u$ and $v$ are tangent to $\mathcal{Q}_o$ at $p_o,$ and then that $[\xi,[u,v]]$ belongs to $p_o$ if $\xi$ belongs to $p_o.$
\\

Finally, we provide another useful formula for the curvature of the tautological bundle $\tau_N\rightarrow\mathcal{Q}_o$:
\begin{lemm}
Let $\omega_o\in\Omega^2(\mathcal{Q}_o,End(\tau_N))$ be the curvature of $\tau_N\rightarrow\mathcal{Q}_o.$ Then, for $u,v\in T_{p}\mathcal{Q}_o,$
$$\omega_o(u,v)=u\circ v^*-v\circ u^*,$$
where $u,v:p\rightarrow p^\perp$ are regarded as linear maps and $u^*,v^*:p^{\perp}\rightarrow p$ are their adjoint.
\end{lemm}
\begin{proof} 
We first note that the adjoint map $u^*:p^{\perp}\rightarrow p$ is explicitly given in terms of the Clifford product by the formula 
$$u^*(\xi)=-\frac{\epsilon_n}{2}\left(\xi\cdot p\cdot u+u\cdot p\cdot \xi\right),  \, \, \,  \forall  \xi \in  p^{\perp};$$ This is the same formula than the formula for $u:p\rightarrow p^{\perp}$ given in Lemma \ref{lemmeClifford}. A straightforward computation then gives
\begin{eqnarray*}
(u\circ v^*-v\circ u^*)(\xi)&=&\frac{1}{4}\left\{(\xi\cdot p\cdot v+v\cdot p\cdot\xi)\cdot p\cdot u+ u\cdot p\cdot (\xi\cdot p\cdot v+v\cdot p\cdot\xi)\right.\\
&&\left.-(\xi\cdot p\cdot u+u\cdot p\cdot\xi)\cdot p\cdot v-v\cdot p\cdot(\xi\cdot p\cdot u+u\cdot p\cdot \xi)\right\}
\end{eqnarray*}
which simplifies to
\begin{equation*}
(u\circ v^*-v\circ u^*)(\xi)=\frac{1}{4}\left\{\xi\cdot p\cdot v\cdot p\cdot u+ u\cdot p\cdot v\cdot p\cdot\xi-\xi\cdot p\cdot u\cdot p\cdot v-v\cdot p\cdot u\cdot p\cdot \xi\right\}.
\end{equation*}
Now, we have $p^2=-\epsilon_n$ and
\begin{eqnarray*}    p\cdot u+u\cdot p=  0\\ p\cdot v+v\cdot p = 0 \end{eqnarray*}
$\forall u,v\in T_p\mathcal{Q}_o$, and we get
\begin{eqnarray*}
(u\circ v^*-v\circ u^*)(\xi)&=&\frac{\epsilon_n}{4}\left\{\xi\cdot(u\cdot v-v\cdot u)-(u\cdot v-v\cdot u)\cdot\xi\right\}\\
&=&\epsilon_n[\xi,[u,v]].
\end{eqnarray*}
This is the expression of the curvature of the bundle $\tau_N\rightarrow\mathcal{Q}_o$ given in Theorem~\ref{th expr curv}.
\end{proof}
We immediately deduce the following 
\begin{coro}\label{coro app normal curvature}
For a smooth map $\overline{\varphi}_o:\M\rightarrow\mathcal{Q}_o,$ we have
$$\overline{\varphi}_o^*\omega_o(X,Y)=d\overline{\varphi}_o(X)\circ d\overline{\varphi}_o(Y)^*-d\overline{\varphi}_o(Y)\circ d\overline{\varphi}_o(X)^*,$$
$\forall X,Y\in T\M,$ where $\overline{\varphi}_o^*\omega_o(X,Y)$ belonging to $\Lambda^2\overline{\varphi}_o^{\perp}$ is regarded as a map $\overline{\varphi}_o^{\perp}\rightarrow\overline{\varphi}_o^{\perp}.$
\end{coro}

%%%%%%%%%%%%%%%%%%%%%%%%%%

\section{The abstract shape operator $B$ identifies to the differential of the Gauss map}\label{app B Gauss map}
The aim is to link the tensor $B$ to the differential of the Gauss map. The next lemma first shows that $B$ is the pull-back of a natural tensor on the tautological bundles on the Grassmannian $\mathcal{Q}_o:$
\begin{lemm} \label{Bprime}
Let us define
\begin{eqnarray*}
B':\hspace{1cm} \tau_N&\rightarrow& T^*\mathcal{Q}_o\otimes\tau_T\\
\xi&\mapsto&Y\mapsto -(d\xi(Y))^T,
\end{eqnarray*}
where, in the right hand side, $\xi$ is extended to a local section of $\tau_N\rightarrow\mathcal{Q}_o.$ We have:
\begin{itemize}
\item[(i)] $B'$ is a tensor; precisely,
$$B'(\xi)(Y)=Y^*(\xi),$$
for all $ \xi\in {\tau_N}_{p_o}\simeq p_o^{\perp}$ and $Y\in T_{p_o}\mathcal{Q}_o\simeq L(p_o,p_o^{\perp}),$ where $Y$ is considered as a linear map $p_o\rightarrow p_o^{\perp}$ and $Y^*: p_o^{\perp}\rightarrow p_o$ is its adjoint.
\item[(ii)] $B$ is the pull-back of $B'$ by the Gauss map:
$$B={\ove{\varphi}_o}^*B'.$$
\end{itemize}
\end{lemm}

\begin{proof}[Proof of (i)] Let $p_o\in\mathcal{Q}_o$ be an oriented $n$-plane, and $\xi\in p_o^\perp$ a vector normal to $p_o,$ extended to a local section of $\tau_N\rightarrow\mathcal{Q}_o.$ Let us consider $* p_o,$ the multi-vector belonging to $\Lambda^k\R^m$ which represents the linear space $p_o^\perp,$ with its natural orientation. We first observe that
$$(d\xi(Y))^T=i_{*p_o}\left(*p_o\wedge d\xi(Y)\right), \, \, \, \forall Y\in T_{p_o}\mathcal{Q}_o.$$ Since $*p_o\wedge\xi\equiv 0$ on $\mathcal{Q}_o,$ we have
$$*p_o\wedge d\xi(Y)=-(*Y)\wedge\xi,$$
and thus
$$-(d\xi(Y))^T=i_{*p_o}\left((*Y)\wedge \xi\right).$$
We finally observe that this expression is $Y^*(\xi),$ where $Y^*$ is the adjoint of the map represented by $Y$: let $Y_{ij}$ be the scalar such that 
$$Y(u_j)=\sum_{i=1}^kY_{ij}\, u_{n+i},\, \, \, \, \, \, \forall j, \, 1 \leq j \leq  n,$$
where $(u_1,\ldots,u_n)$ and $(u_{n+1},\ldots,u_{n+k})$ are  positively oriented, orthonormal bases of $p_o$ and $p_o^{\perp}$ respectively. We have
\begin{eqnarray*}
Y&=&\sum_{j=1}^n u_1\wedge\ldots\wedge u_{j-1}\wedge Y(u_j)\wedge u_{j+1}\wedge\ldots\wedge u_n\\
&=&\sum_{i=1}^k\sum_{j=1}^nY_{ij}\,  u_1\wedge\ldots\wedge u_{j-1}\wedge u_{n+i}\wedge u_{j+1}\wedge\ldots\wedge u_n
\end{eqnarray*}
and
\begin{eqnarray*}
*Y&=&\sum_{i=1}^k\sum_{j=1}^n Y_{ij}\ *\left(u_1\wedge\ldots\wedge u_{j-1}\wedge u_{n+i}\wedge u_{j+1}\wedge\ldots\wedge u_n\right)\\
&=&-\sum_{i=1}^k\sum_{j=1}^n Y_{ij}\ u_{n+1}\wedge\ldots\wedge u_{n+i-1}\wedge u_{j}\wedge u_{n+i+1}\wedge\ldots\wedge u_{n+k}.
\end{eqnarray*}
Since $*p_o=u_{n+1}\wedge\ldots\wedge u_{n+k},$ we get
$$i_{*p_o}\left((*Y)\wedge u_{n+i}\right)=\sum_{j=1}^n Y_{ij}\, u_j, \, \, \, \, \, \,  \forall j, \, 1 \leq j \leq  k,$$
and the claim follows.
\\
\\ \noindent \textit{Proof of (ii).} Assume that $\xi\in\Gamma(\tau_N);$ then $\xi\circ\ove{\varphi}_o\in\Gamma(\E),$ and
\begin{eqnarray*}
B(\xi\circ\overline{\varphi}_o)(X)&=&(d(\xi\circ\ove{\varphi}_o)(X))^T \\ 
&=& \big(d\xi(d\ove{\varphi}_o(X)) \big)^T\\
&=&B'(\xi)(d\ove{\varphi}_o(X)) \\
&=&\ove{\varphi}_o^*B'(\xi\circ\ove{\varphi}_o)(X), \quad \quad \forall X\in\Gamma(T\M).
\end{eqnarray*}
\end{proof}
Using Lemma \ref{Bprime}, we deduce that if $\xi$ belongs to the fibre of $ {\overline{\varphi}_o}^*\tau_N$ at the point $x_o\in\M,$ then
$$B(\xi)(X)=d {\overline{\varphi}_o}_{x_o}(X)^*(\xi), \quad \forall  X\in T_{x_o}\M,$$ 
where  $d {\ove{\varphi}_o}_{x_o}(X)^*$ is regarded as a map $\ove{\varphi}_o(x_o)^{\perp}\rightarrow \ove{\varphi}_o(x_o)$; this naturally identifies $B$ with $d{\ove{\varphi}}_o.$


\begin{thebibliography}{XXXXX}
\bibitem[An]{An}  H. Anciaux, \em Spaces of geodesics of pseudo-Riemannian space forms and normal congruences of hypersurfaces, \em  Trans. of the AMS \textbf{366} (2014) 2699--2718. 
\bibitem[CL]{CL} W. Chen, H. Li, {\it The Gauss map of space-like surfaces in $\R_p^{2+p}$,} Kyushu J.\ Math.\ \textbf{51} (1997)  217--224.
\bibitem[F]{F} Th. Friedrich, {\it Dirac Operators in Riemannian Geometry}, Graduate studies in Mathematics \textbf{25}, AMS. 
\bibitem[GK]{GK}  B. Guilfoyle, W. Klingenberg, \em Generalised surfaces in $\R^3,$ \em Math. Proc. R. Ir. Acad. {\bf 104A:2} (2004) 199--209.
\bibitem[Ku]{Ku}  W. K\"uhnel, \em Differential Geometry, Curves -- Surfaces -- Manifolds, \em Student Mathematical Library \textbf{16}, AMS.
\bibitem[HO1]{HO1} D. Hoffman, R. Osserman, {\it The Gauss map of surfaces in $\R^n$}, J.\ Differential Geometry {\bf 18} (1983) 733--754.
\bibitem[HO2]{HO2} D. Hoffman, R. Osserman, {\it The Gauss map of surfaces in $\R^3$ and $\R^4$}, Proc.\ London Math. Soc. {\bf s3-50:1} (1985) 27--56.
%\bibitem[Sa]{Sa} M. Salvai, {\it On the geometry of the space of oriented lines of Euclidean space}, Manuscripta Math. \textbf{118:2} (2005) 181--189
\bibitem[W1]{W1} J.L. Weiner, {\it The Gauss map for surfaces in 4-space}, Math.\ Ann.\ {\bf 269} (1984) 541--560. 
\bibitem[W2]{W2} J.L. Weiner,  {\it The Gauss map for surfaces: Part 1. The affine case}, Trans. of the AMS \textbf{293:2} (1986) 431--446.
\bibitem[W3]{W3} J.L. Weiner, {\it The Gauss map for surfaces: Part 2. The euclidean case}, Trans. of the AMS \textbf{293:2} (1986) 447--466.
\end{thebibliography}
\end{document}